\documentclass[12pt]{article}

\usepackage{amsmath, epsfig, cite}
\usepackage{amssymb}
\usepackage{amsfonts}
\usepackage{latexsym}
\usepackage{amsthm}

\newtheorem{thm}{Theorem}[section]

\numberwithin{equation}{section}

\renewcommand{\thefootnote}

\begin{document}

\begin{center}
{\large\bf On some conjectural series containing binomial
coefficients\\ and harmonic numbers
 \footnote{ The work is supported by the National Natural Science Foundation of China (No. 12071103).}}
\end{center}

\renewcommand{\thefootnote}{$\dagger$}

\vskip 2mm \centerline{Chuanan Wei}
\begin{center}
{School of Biomedical Information and Engineering\\ Hainan Medical
University, Haikou 571199, China
\\
 Email address: weichuanan78@163.com}
\end{center}


\vskip 0.7cm \noindent{\bf Abstract.} Binomial coefficients and
harmonic numbers are important in many branches of number theory.
With the help of the operator method and several summation and
transformation formulas for hypergeometric series, we prove eight
conjectural series of Z.-W. Sun containing binomial coefficients and
harmonic numbers in this paper.



\section{Introduction}

For $\ell,n\in \mathbb{Z}^{+}$, define the generalized $\ell$-order
harmonic numbers as
\[H_{n}^{(\ell)}(x)=\sum_{k=1}^n\frac{1}{(x+k)^{\ell}},\]
where  $x$ is a complex variable. The $x=0$ case of them are the
$\ell$-order harmonic numbers
\[H_{n}^{(\ell)}=\sum_{k=1}^n\frac{1}{k^{\ell}}.\]
When $\ell=0$, they reduce to classical harmonic numbers:
\[
H_{n}=\sum_{k=1}^n\frac{1}{k}.\] For a nonnegative integer $m$,
define the shifted-factorial to be
\begin{align*}
(x)_0=1 \quad\text{and}\quad (x)_m=x(x+1)\cdots(x+m-1) \quad
\text{when} \quad m\in \mathbb{Z}^{+}.
\end{align*}
 For a differentiable function $f(x)$, define the derivative operator
$\mathcal{D}_x$ by
\begin{align*}
\mathcal{D}_xf(x)=\frac{d}{dx}f(x).
 \end{align*}
 Then it is ordinary to find that
\begin{align*}
\mathcal{D}_x\:(1+x)_r=(1+x)_rH_r(x).
 \end{align*}
Several nice harmonic number identities from differentiation of the
shifted-factorials can be viewed in the papers
\cite{Liu,Paule,Sofo,Wang-Wei}.

Define the digamma function $\psi(x)$ as
\begin{align*}
\psi(x)=\frac{d}{dx}\big\{\log\Gamma(x)\big\},
\end{align*}
where $\Gamma(x)$ is the familiar gamma function. Furthermore, we
can define the polygamma function $\psi^{(n)}(x)$ to be
\begin{align*}
\psi^{(n)}(x)=\frac{d^{n+1}}{dx^{n+1}}\big\{\log\Gamma(x)\big\}=\frac{d^{n}}{dx^{n}}\psi(x).
 \end{align*}
It is famous that the polygamma function satisfies the recurrence
relation:
\begin{align}
\psi^{(n)}(x+1)=\psi^{(n)}(x)+\frac{(-1)^nn!}{x^{n+1}}.
\label{recurrence}
 \end{align}
Two related special values of $\psi\,'(x)$ (cf. \cite{Liu}) read
\begin{align}
&\psi\,'\bigg(\frac{1}{4}\bigg)=\pi^2+8G,
 \label{digamma-b}\\[2mm]
&\psi\,'\bigg(\frac{3}{4}\bigg)=\pi^2-8G, \label{digamma-c}
\end{align}
where $G$ is the Catalan constant:
$$G=\sum_{k=0}^{\infty}\frac{(-1)^k}{(1+2k)^2}.$$

Sun \cite[Corollary 1.4]{Sun-a} provided the following
supercongruence:
\begin{align*}
&\sum_{k=0}^{p-1}\frac{\binom{2k}{k}\binom{3k}{k}}{(-216)^k}\equiv\bigg(\frac{p}{3}\bigg)\sum_{k=0}^{p-1}\frac{\binom{2k}{k}\binom{3k}{k}}{24^k}\pmod{p^2},
\\[1mm]
&\sum_{k=0}^{p-1}\frac{\binom{2k}{k}\binom{4k}{2k}}{(-192)^k}\equiv\bigg(\frac{-2}{p}\bigg)\sum_{k=0}^{p-1}\frac{\binom{2k}{k}\binom{4k}{2k}}{48^k}\pmod{p^2},
\\[1mm]
&\sum_{k=0}^{p-1}\frac{\binom{2k}{k}\binom{4k}{2k}}{(-4032)^k}\equiv\bigg(\frac{-2}{p}\bigg)\sum_{k=0}^{p-1}\frac{\binom{2k}{k}\binom{4k}{2k}}{63^k}\pmod{p^2},
\\[1mm]
&\sum_{k=0}^{p-1}\frac{\binom{2k}{k}\binom{4k}{2k}}{576^k}\equiv\bigg(\frac{-2}{p}\bigg)\sum_{k=0}^{p-1}\frac{\binom{2k}{k}\binom{4k}{2k}}{72^k}\pmod{p^2},
\end{align*}
where $p>3$ is any prime and $(\frac{.}{p})$ denotes the Legendre
symbol. Some related supercongruences can be viewed in the papers
\cite{Mao,Sun-11,Wang-Sun}.

Motivated by the works just mentioned, Sun \cite[Equations (2.21),
(2.25), (2.26), (2.23), (2.24)]{Sun} proposed the following five
conjectures containing binomial coefficients and harmonic numbers.

\begin{thm}\label{thm-a}
\begin{align}
&\sum_{k=0}^{\infty}\frac{\binom{2k}{k}\binom{3k}{k}}{(-216)^k}(3H_{3k}-H_k)=\bigg(\log\frac{8}{9}\bigg)\sum_{k=0}^{\infty}\frac{\binom{2k}{k}\binom{3k}{k}}{(-216)^k},
\label{equation-wei-a}\\[1mm]
&\sum_{k=0}^{\infty}\frac{\binom{2k}{k}\binom{4k}{2k}}{(-192)^k}(2H_{4k}-H_{2k})=\frac{1}{2}\bigg(\log\frac{3}{4}\bigg)\sum_{k=0}^{\infty}\frac{\binom{2k}{k}\binom{4k}{2k}}{(-192)^k},
\label{equation-wei-b}\\[1mm]
&\sum_{k=0}^{\infty}\frac{\binom{2k}{k}\binom{4k}{2k}}{(-4032)^k}(2H_{4k}-H_{2k})=\frac{1}{2}\bigg(\log\frac{63}{64}\bigg)\sum_{k=0}^{\infty}\frac{\binom{2k}{k}\binom{4k}{2k}}{(-4032)^k},
\label{equation-wei-c}\\[1mm]
&\sum_{k=0}^{\infty}\frac{\binom{2k}{k}\binom{4k}{2k}}{72^k}(2H_{4k}-H_{2k})=(\log3)\sum_{k=0}^{\infty}\frac{\binom{2k}{k}\binom{4k}{2k}}{72^k},
\label{equation-wei-d}\\[1mm]
&\sum_{k=0}^{\infty}\frac{\binom{2k}{k}\binom{4k}{2k}}{576^k}(2H_{4k}-H_{2k})=\frac{1}{2}\bigg(\log\frac{9}{8}\bigg)\sum_{k=0}^{\infty}\frac{\binom{2k}{k}\binom{4k}{2k}}{576^k}.
\label{equation-wei-e}
\end{align}
\end{thm}

Motivated by the series from {\bf{Mathematica}}:
\begin{align*}
\sum_{k=0}^{\infty}\frac{\binom{2k}{k}^2}{32^k}=\frac{\Gamma(1/4)^2}{2\pi\sqrt{\pi}},
\end{align*}
Sun \cite[Equation (2.19)]{Sun} conjectured the following series
containing binomial coefficients and harmonic numbers.

\begin{thm}\label{thm-b}
\begin{align}
\sum_{k=0}^{\infty}\frac{\binom{2k}{k}^2}{32^k}\bigg\{H_{2k}^{(2)}-\frac{1}{4}H_{k}^{(2)}\bigg\}=\Gamma\bigg(\frac{1}{4}\bigg)^2\frac{\pi^2-8G}{32\pi\sqrt{\pi}}.
 \label{equation-wei-f}
\end{align}
\end{thm}

There exist a lot of interesting $\pi$-formulas in the literature.
Two series for $1/\pi^2$ due to Guillera
\cite{Guillera-a,Guillera-c} can be laid out as follows:
\begin{align}
&\quad\sum_{k=0}^{\infty}(20k^2+8k+1)\frac{\binom{2k}{k}^5}{(-2^{12})^k}=\frac{8}{\pi^2},
\label{Guillera-a}\\[1mm]
&\sum_{k=0}^{\infty}(820k^2+180k+13)\frac{\binom{2k}{k}^5}{(-2^{20})^k}=\frac{128}{\pi^2}.
\label{Guillera-b}
\end{align}
For more conclusions on $\pi$-formulas, the reader is referred to
 the papers \cite{Au,Guo,Liu-a,Sun-b,Wang}.

Encouraged by \eqref{Guillera-a} and \eqref{Guillera-b},  Sun
\cite[Equations (4.19) and (4.24)]{Sun} proposed the following two
conjectures containing binomial coefficients and harmonic numbers.

\begin{thm}\label{thm-c}
\begin{align}
&\quad\:\:\sum_{k=0}^{\infty}\frac{\binom{2k}{k}^5}{(-2^{12})^k}\Big\{(20k^2+8k+1)\big[8H_{2k}^{(2)}-3H_{k}^{(2)}\big]+4\Big\}=\frac{8}{3},
 \label{equation-wei-g}\\[1mm]
&\sum_{k=0}^{\infty}\frac{\binom{2k}{k}^5}{(-2^{20})^k}\Big\{(820k^2+180k+13)\big[11H_{2k}^{(2)}-3H_{k}^{(2)}\big]+43\Big\}=\frac{128}{3}.
 \label{equation-wei-h}
\end{align}
\end{thm}

The rest of the paper is organized as follows. According to the
operator method and several summation and transformation formulas
for hypergeometric series, we shall certify Theorems
\ref{thm-a}-\ref{thm-c} in Sections 2-4, respectively.

\section{Proof of Theorem \ref{thm-a}}
Above all, we shall give the following parametric generalizations of
\eqref{equation-wei-a}-\eqref{equation-wei-e}.

\begin{thm}\label{thm-d}
 Let $x$ be a complex number. Then
\begin{align}
&\sum_{k=0}^{\infty}\frac{\binom{2k}{k}^2}{x^k}(2H_{2k}-H_k)=\frac{1}{2}\bigg(\log\frac{x}{x-16}\bigg)\sum_{k=0}^{\infty}\frac{\binom{2k}{k}^2}{x^k},
\label{equation-wei-o}
\end{align}
where $|x|>16$,
\begin{align}
&\sum_{k=0}^{\infty}\frac{\binom{2k}{k}\binom{3k}{k}}{x^k}(3H_{3k}-H_k)=\bigg(\log\frac{x}{x-27}\bigg)\sum_{k=0}^{\infty}\frac{\binom{2k}{k}\binom{3k}{k}}{x^k},
\label{equation-wei-p}
\end{align}
where $|x|>27$,
\begin{align}
&\sum_{k=0}^{\infty}\frac{\binom{2k}{k}\binom{4k}{2k}}{x^k}(2H_{4k}-H_{2k})=\frac{1}{2}\bigg(\log\frac{x}{x-64}\bigg)\sum_{k=0}^{\infty}\frac{\binom{2k}{k}\binom{4k}{2k}}{x^k},
\label{equation-wei-q}
\end{align}
where $|x|>64$,
\begin{align}
&\sum_{k=0}^{\infty}\frac{\binom{3k}{k}\binom{6k}{3k}}{x^k}(6H_{6k}-3H_{3k}-2H_{2k}+H_k)=\bigg(\log\frac{x}{x-432}\bigg)\sum_{k=0}^{\infty}\frac{\binom{3k}{k}\binom{6k}{3k}}{x^k},
\label{equation-wei-r}
\end{align}
where $|x|>432$.
\end{thm}

\begin{proof}
Following Bailey \cite{Bailey}, define the hypergeometric by
$$
_{r+1}F_{r}\left[\begin{array}{c}
a_1,a_2,\ldots,a_{r+1}\\
b_1,b_2,\ldots,b_{r}
\end{array};\, z
\right] =\sum_{k=0}^{\infty}\frac{(a_1)_k(a_2)_k\cdots(a_{r+1})_k}
{(1)_k(b_1)_k\cdots(b_{r})_k}z^k.
$$
Then Euler's transformation formula connecting two $_2F_1$ series
(cf. \cite[P. 2]{Bailey}) may be stated as
\begin{align}
{_{2}F_{1}}\left[\begin{array}{cccccccc}
  a,b\\ c
\end{array};x\right]
=(1-x)^{c-a-b}{_{2}F_{1}}\left[\begin{array}{cccccccc}
  c-a,c-b\\ c
\end{array};x\right],
 \label{eq:Euler}
\end{align}
where $|x|<1$. It is routine to understand that the two series in
\eqref{eq:Euler}
 are both uniformly convergent for $a\in \mathbb{C}$. Apply the operator
$\mathcal{D}_{a}$ on both sides of \eqref{eq:Euler} to discover
\begin{align*}
&\sum_{k=0}^{\infty}\frac{(a)_k(b)_k}{(1)_{k}(c)_{k}}x^kH_k(a-1)
\notag\\[1mm]
&\:\:=-(1-x)^{c-a-b}\{\log(1-x)\}\sum_{k=0}^{\infty}\frac{(c-a)_k(c-b)_k}{(1)_{k}(c)_{k}}x^k
\notag\\[1mm]
&\:\:\quad-(1-x)^{c-a-b}\sum_{k=0}^{\infty}\frac{(c-a)_k(c-b)_k}{(1)_{k}(c)_{k}}x^kH_k(c-a-1).
\end{align*}
The $c=a+b$ case of it becomes
\begin{align}
&\sum_{k=0}^{\infty}\frac{(a)_k(b)_k}{(1)_{k}(a+b)_{k}}x^kH_k(a-1)
\notag\\[1mm]
&\:\:=-\{\log(1-x)\}\sum_{k=0}^{\infty}\frac{(a)_k(b)_k}{(1)_{k}(a+b)_{k}}x^k
-\sum_{k=0}^{\infty}\frac{(a)_k(b)_k}{(1)_{k}(a+b)_{k}}x^kH_k(b-1).
\label{equation-wei-s}
\end{align}
Replacing $x$ by $1/x$, equation \eqref{equation-wei-s} can be
manipulated as
\begin{align}
\sum_{k=0}^{\infty}\frac{(a)_k(b)_k}{(1)_{k}(a+b)_{k}}\frac{H_k(a-1)+H_k(b-1)}{x^k}
=\bigg(\log\frac{x}{x-1}\bigg)\sum_{k=0}^{\infty}\frac{(a)_k(1-a)_k}{(1)_{k}(a+b)_{k}}\frac{1}{x^k}.
\label{equation-wei-t}
\end{align}

Choosing $(a, b, x)\mapsto(\frac{1}{2}, \frac{1}{2}, \frac{x}{16})$
in \eqref{equation-wei-t} and using the following two relations:
\begin{align*}
&\qquad\quad\frac{(\frac{1}{2})_k^2}{(1)_{k}^2}=\frac{\binom{2k}{k}^2}{16^k},
\\[1mm]
&H_k\bigg(-\frac{1}{2}\bigg)=2H_{2k}-H_k,
\end{align*}
we obtain \eqref{equation-wei-o}. Fixing $(a, b,
x)\mapsto(\frac{1}{3}, \frac{2}{3}, \frac{x}{27})$ in
\eqref{equation-wei-t} and utilizing the following two relations:
\begin{align*}
&\qquad\qquad\frac{(\frac{1}{3})_k(\frac{2}{3})_k}{(1)_{k}^2}=\frac{\binom{2k}{k}\binom{3k}{k}}{27^k},
\\[1mm]
&H_k\bigg(-\frac{1}{3}\bigg)+H_k\bigg(-\frac{2}{3}\bigg)=3H_{3k}-H_k,
\end{align*}
we deduce \eqref{equation-wei-p}. Setting $(a, b,
x)\mapsto(\frac{1}{4}, \frac{3}{4}, \frac{x}{64})$ in
\eqref{equation-wei-t} and using the following two relations:
\begin{align*}
&\qquad\qquad\frac{(\frac{1}{4})_k(\frac{3}{4})_k}{(1)_{k}^2}=\frac{\binom{2k}{k}\binom{4k}{2k}}{64^k},
\\[1mm]
&H_k(-\tfrac{1}{4})+H_k(-\tfrac{3}{4})=4H_{4k}-2H_{2k},
\end{align*}
we arrive at \eqref{equation-wei-q}. Taking $(a, b,
x)\mapsto(\frac{1}{6}, \frac{5}{6}, \frac{x}{432})$ in
\eqref{equation-wei-t} and utilizing the following two relations:
\begin{align*}
&\qquad\qquad\qquad\frac{(\frac{1}{6})_k(\frac{5}{6})_k}{(1)_{k}^2}=\frac{\binom{3k}{k}\binom{6k}{3k}}{432^k},
\\[1mm]
&H_k(-\tfrac{1}{6})+H_k(-\tfrac{5}{6})=6H_{6k}-3H_{3k}-2H_{2k}+H_k,
\end{align*}
we are led to \eqref{equation-wei-r}.
\end{proof}

Now we are ready to prove Theorem \ref{thm-a}.

\begin{proof}[Proof of Theorem \ref{thm-a}]
The $x=-216$ case of \eqref{equation-wei-p} is
\eqref{equation-wei-a}. Selecting $x=-192,\,-4032,\,72$ and $576$ in
\eqref{equation-wei-q}, we catch hold of \eqref{equation-wei-b},
\eqref{equation-wei-c}, \eqref{equation-wei-d} and
\eqref{equation-wei-e}, respectively.
\end{proof}
\section{Proof of Theorem \ref{thm-b}}
For the goal of proving Theorem \ref{thm-b}, we require Bailey's
$_2F_1$ summation formula (cf. \cite[P. 17]{Bailey}):
\begin{align}
&{_{2}F_{1}}\left[\begin{array}{cccccccc}
  a,1-a\\
 b
\end{array};\frac{1}{2}\right]
=\frac{\Gamma(\frac{b}{2})\Gamma(\frac{1+b}{2})}{\Gamma(\frac{a+b}{2})\Gamma(\frac{1-a+b}{2})}.
 \label{eq:Bailey}
\end{align}

Now we begin to prove Theorem \ref{thm-b}.

\begin{proof}[{\bf{Proof of Theorem \ref{thm-b}}}]
We know that the series in \eqref{eq:Bailey} is uniformly convergent
for $a\in \mathbb{C}$. Employ $\mathcal{D}_{a}$ on both sides of
\eqref{eq:Bailey} to get
\begin{align}
&\sum_{k=0}^{\infty}\bigg(\frac{1}{2}\bigg)^{k-1}\frac{(a)_k(1-a)_k}{(1)_{k}(b)_{k}}\{H_k(a-1)-H_k(-a)\}
=\frac{\Gamma(\frac{b}{2})\Gamma(\frac{1+b}{2})}{\Gamma(\frac{a+b}{2})\Gamma(\frac{1-a+b}{2})}
\notag\\[1mm]
&\:\:\times
\bigg\{\psi\bigg(\frac{1-a+b}{2}\bigg)-\psi\bigg(\frac{a+b}{2}\bigg)\bigg\}.
 \label{eq:Bailey-a}
\end{align}
Via the relation:
\begin{align*}
H_k(a-1)-H_k(-a)=\sum_{i=1}^k\frac{1}{a-1+i}-\sum_{i=1}^k\frac{1}{-a+i}=\sum_{i=1}^k\frac{1-2a}{(a-1+i)(-a+i)},
\end{align*}
equation  \eqref{eq:Bailey-a} can be reformulated as
\begin{align}
&\sum_{k=0}^{\infty}\bigg(\frac{1}{2}\bigg)^{k-1}\frac{(a)_k(1-a)_k}{(1)_{k}(b)_{k}}\sum_{i=1}^k\frac{1}{(a-1+i)(-a+i)}
=\frac{\Gamma(\frac{b}{2})\Gamma(\frac{1+b}{2})}{\Gamma(\frac{a+b}{2})\Gamma(\frac{1-a+b}{2})}
\notag\\[1mm]
&\:\:\times \frac{\psi(\frac{1-a+b}{2})-\psi(\frac{a+b}{2})}{1-2a}.
 \label{eq:Bailey-b}
\end{align}
 By the L'H\^{o}pital rule, we have
\begin{align*}
\lim_{a\to\frac{1}{2}}\frac{\psi(\frac{1-a+b}{2})-\psi(\frac{a+b}{2})}{1-2a}
=\frac{\psi\,'(\frac{1+2b}{4})}{2}.
\end{align*}
Letting $a\to\frac{1}{2}$ in \eqref{eq:Bailey-b} and using the upper
limit, there is
\begin{align}
&\sum_{k=0}^{\infty}\bigg(\frac{1}{2}\bigg)^{k}\frac{(\frac{1}{2})_k^2}{(1)_{k}(b)_{k}}\bigg\{H_{2k}^{(2)}-\frac{1}{4}H_{k}^{(2)}\bigg\}
=\frac{\Gamma(\frac{b}{2})\Gamma(\frac{1+b}{2})}{16\Gamma(\frac{1+2b}{4})^2}\psi\,'\bigg(\frac{1+2b}{4}\bigg).
 \label{eq:Bailey-c}
\end{align}
Choosing $b=1$ in  \eqref{eq:Bailey-c} and using \eqref{digamma-b},
we find \eqref{equation-wei-f}.
\end{proof}

In this section, we shall also establish the following theorem
similar to Theorem \ref{thm-b}.

\begin{thm}\label{thm-d}
\begin{align}
\sum_{k=0}^{\infty}\frac{\binom{2k}{k}^2}{32^k(1+k)}\bigg\{H_{2k}^{(2)}-\frac{1}{4}H_{k}^{(2)}\bigg\}=\Gamma\bigg(\frac{3}{4}\bigg)^2\frac{\pi^2+8G-16}{4\pi\sqrt{\pi}}.
 \label{wei-g}
\end{align}
\end{thm}

\begin{proof}
Fixing $b=2$ in  \eqref{eq:Bailey-c}, there is
\begin{align}
&\sum_{k=0}^{\infty}\frac{\binom{2k}{k}^2}{32^k(1+k)}\bigg\{H_{2k}^{(2)}-\frac{1}{4}H_{k}^{(2)}\bigg\}
=\frac{\Gamma(\frac{3}{2})}{16\Gamma(\frac{5}{4})^2}\psi\,'\bigg(\frac{5}{4}\bigg).
 \label{eq:Bailey-d}
\end{align}
According to \eqref{recurrence} and \eqref{digamma-b}, it is easy to
show that
\begin{align}
&\psi\,'\bigg(\frac{5}{4}\bigg)=\psi\,'\bigg(\frac{1}{4}\bigg)-16=\pi^2+8G-16.
 \label{eq:Bailey-e}
\end{align}
So the combination of \eqref{eq:Bailey-d} and \eqref{eq:Bailey-e}
produces \eqref{wei-g}.
\end{proof}
\section{Proof of Theorem
\ref{thm-c}}

In order to prove Theorem \ref{thm-c}, we need Dougall's $_5F_4$
 summation formula (cf. \cite[P. 27]{Bailey}):
\begin{align}
&{_{5}F_{4}}\left[\begin{array}{cccccccc}
  a,1+\frac{a}{2},b,c,d\\
  \frac{a}{2},1+a-b,1+a-c,1+a-d
\end{array};1\right]
\notag\\[1mm]
&\:\:=
\frac{\Gamma(1+a-b)\Gamma(1+a-c)\Gamma(1+a-d)\Gamma(1+a-b-c-d)}{\Gamma(1+a)\Gamma(1+a-b-c)\Gamma(1+a-b-d)\Gamma(1+a-c-d)},
 \label{eq:Dougall}
\end{align}
where  $\mathfrak{R}(1+a-b-c-d)>0$.

Now we begin to prove Theorem \ref{thm-c}.

\begin{proof}[{\bf{Proof of Theorem \ref{thm-c}}}]
Recall the following transformation formula for hypergeometric
series (cf. \cite[Theorem 9]{Chu-b}):
\begin{align}
&\sum_{k=0}^{\infty}\frac{(c)_k(d)_k(e)_k(1+a-b-c)_k(1+a-b-d)_{k}(1+a-b-e)_{k}}{(1+a-c)_{k}(1+a-d)_{k}(1+a-e)_{k}(1+2a-b-c-d-e)_{k}}
\notag\\[1mm]
&\quad\times\frac{(-1)^k}{(1+a-b)_{2k}}\alpha_k(a,b,c,d,e)
\notag\\[1mm]
&\:=\sum_{k=0}^{\infty}(a+2k)\frac{(b)_k(c)_k(d)_k(e)_k}{(1+a-b)_{k}(1+a-c)_{k}(1+a-d)_{k}(1+a-e)_{k}},
\label{equation-a}
\end{align}
where  $\mathfrak{R}(1+2a-b-c-d-e)>0$ and
\begin{align*}
\alpha_k(a,b,c,d,e)&=\frac{(1+2a-b-c-d+2k)(a-e+k)}{1+2a-b-c-d-e+k}
\\[1mm]
&\quad+\frac{(1+a-b-c+k)(1+a-b-d+k)(e+k)}{(1+a-b+2k)(1+2a-b-c-d-e+k)}.
\end{align*}

Setting $e=a$ in \eqref{equation-a} and calculating the series on
the right-hand side by \eqref{eq:Dougall}, we discover
\begin{align*}
&\sum_{k=0}^{\infty}\frac{(a)_k(c)_k(d)_k(1-b)_{k}(1+a-b-c)_k(1+a-b-d)_{k}}{(1)_{k}(1+a-c)_{k}(1+a-d)_{k}(2+a-b-c-d)_{k}}
\notag\\[1mm]
&\quad\times\frac{(-1)^k}{(1+a-b)_{2k}}\beta_k(a,b,c,d)
\notag\\[1mm]
&\:=\frac{\Gamma(1+a-b)\Gamma(1+a-c)\Gamma(1+a-d)\Gamma(2+a-b-c-d)}{\Gamma(1+a)\Gamma(1+a-b-c)\Gamma(1+a-b-d)\Gamma(1+a-c-d)},
\end{align*}
where
\begin{align*}
\beta_k(a,b,c,d)&=\frac{k(1+2a-b-c-d+2k)}{a}
\\[1mm]
&\quad+\frac{(a+k)(1+a-b-c+k)(1+a-b-d+k)}{a(1+a-b+2k)}.
\end{align*}
The $(a,b,c)=(\frac{1}{2},\frac{1}{2},1-d)$ case of it can be
expressed as
\begin{align}
\sum_{k=0}^{\infty}\bigg(\frac{-1}{4}\bigg)^k\frac{(\frac{1}{2})_k(d)_k^2(1-d)_k^2}{(1)_{k}^3(\frac{1}{2}+d)_{k}(\frac{3}{2}-d)_{k}}(d-d^2+2k+5k^2)
=\frac{1-2d}{\pi}\tan(d\pi).\label{equation-b}
\end{align}
Notice that the series in \eqref{equation-b} is uniformly convergent
for $d\in \mathbb{C}$. Apply $\mathcal{D}_{d}$ on both sides of
\eqref{equation-b} to obtain
\begin{align*}
&\sum_{k=0}^{\infty}\bigg(\frac{-1}{4}\bigg)^k\frac{(\frac{1}{2})_k(d)_k^2(1-d)_k^2}{(1)_{k}^3(\frac{1}{2}+d)_{k}(\frac{3}{2}-d)_{k}}(d-d^2+2k+5k^2)
\notag\\[1mm]
&\quad\times\bigg\{2H_{k}(d-1)-2H_{k}(-d)+H_{k}\bigg(\frac{1}{2}-d\bigg)-H_{k}\bigg(d-\frac{1}{2}\bigg)\bigg\}
\notag\\[1mm]
&\:+\sum_{k=0}^{\infty}\bigg(\frac{-1}{4}\bigg)^k\frac{(\frac{1}{2})_k(d)_k^2(1-d)_k^2}{(1)_{k}^3(\frac{1}{2}+d)_{k}(\frac{3}{2}-d)_{k}}(1-2d)
\notag\\[1mm]
&\:\:=(1-2d)\sec^2(d\pi)-\frac{2}{\pi}\tan(d\pi).
\end{align*}
Dividing both sides of the last equation by $(1-2d)$, we have
\begin{align}
&\sum_{k=0}^{\infty}\bigg(\frac{-1}{4}\bigg)^k\frac{(\frac{1}{2})_k(d)_k^2(1-d)_k^2}{(1)_{k}^3(\frac{1}{2}+d)_{k}(\frac{3}{2}-d)_{k}}(d-d^2+2k+5k^2)
\notag\\[1mm]
&\quad\times\bigg\{\sum_{i=1}^k\frac{2}{(d-1+i)(-d+i)}-\sum_{i=1}^k\frac{1}{(d-\frac{1}{2}+i)(\frac{1}{2}-d+i)}\bigg\}
\notag\\[1mm]
&\:+\sum_{k=0}^{\infty}\bigg(\frac{-1}{4}\bigg)^k\frac{(\frac{1}{2})_k(d)_k^2(1-d)_k^2}{(1)_{k}^3(\frac{1}{2}+d)_{k}(\frac{3}{2}-d)_{k}}
\notag\\[1mm]
&\:\:=\sec^2(d\pi)-\frac{2}{\pi}\frac{\tan(d\pi)}{1-2d}.
\label{equation-c}
\end{align}
By the L'H\^{o}pital rule, there holds
\begin{align}
\lim_{d\to\frac{1}{2}}\bigg\{\sec^2(d\pi)-\frac{2}{\pi}\frac{\tan(d\pi)}{1-2d}\bigg\}
=\frac{2}{3}. \label{rule}
\end{align}
Letting $a\to\frac{1}{2}$ in \eqref{equation-c} and using
\eqref{rule}, we deduce \eqref{equation-wei-g}.

Recollect the following transformation formula for hypergeometric
series (cf. \cite[Theorem 32]{Chu-b}):
\begin{align}
&\sum_{k=0}^{\infty}(-1)^k\frac{(b)_k(c)_k(d)_k(e)_k(1+a-b-c)_k(1+a-b-d)_{k}(1+a-b-e)_{k}}{(1+a-b)_{2k}(1+a-c)_{2k}(1+a-d)_{2k}(1+a-e)_{2k}}
\notag\\[1mm]
&\quad\times\frac{(1+a-c-d)_k(1+a-c-e)_{k}(1+a-d-e)_{k}}{(1+2a-b-c-d-e)_{2k}}\lambda_k(a,b,c,d,e)
\notag\\[1mm]
&\:=\sum_{k=0}^{\infty}(a+2k)\frac{(b)_k(c)_k(d)_k(e)_k}{(1+a-b)_{k}(1+a-c)_{k}(1+a-d)_{k}(1+a-e)_{k}},
\label{equation-aa}
\end{align}
where $\mathfrak{R}(1+2a-b-c-d-e)>0$ and
\begin{align*}
&\lambda_k(a,b,c,d,e)\\[1mm]
&\:=\frac{(1+2a-b-c-d+3k)(a-e+2k)}{1+2a-b-c-d-e+2k}+\frac{(e+k)(1+a-b-c+k)}{(1+a-b+2k)(1+a-d+2k)}
\\[1mm]
&\quad\times\frac{(1+a-b-d+k)(1+a-c-d+k)(2+2a-b-d-e+3k)}{(1+2a-b-c-d-e+2k)(2+2a-b-c-d-e+2k)}
\\[1mm]
&\:+\frac{(c+k)(e+k)(1+a-b-c+k)(1+a-b-d+k)}{(1+a-b+2k)(1+a-c+2k)(1+a-d+2k)(1+a-e+2k)}
\\[1mm]
&\quad\times\frac{(1+a-b-e+k)(1+a-c-d+k)(1+a-d-e+k)}{(1+2a-b-c-d-e+2k)(2+2a-b-c-d-e+2k)}.
\end{align*}

Taking $e=a$ in \eqref{equation-aa} and evaluating the series on the
right-hand side by \eqref{eq:Dougall}, we get
\begin{align*}
&\sum_{k=0}^{\infty}(-1)^k\frac{(a)_k(b)_k(c)_k(d)_k(1-b)_k(1-c)_{k}(1-d)_{k}}{(1)_{2k}(1+a-b)_{2k}(1+a-c)_{2k}(1+a-d)_{2k}}
\notag\\[1mm]
&\quad\times\frac{(1+a-b-c)_k(1+a-b-d)_{k}(1+a-c-d)_{k}}{(2+a-b-c-d)_{2k}}\theta_k(a,b,c,d)
\notag\\[1mm]
&\:=\frac{\Gamma(1+a-b)\Gamma(1+a-c)\Gamma(1+a-d)\Gamma(2+a-b-c-d)}{\Gamma(1+a)\Gamma(1+a-b-c)\Gamma(1+a-b-d)\Gamma(1+a-c-d)},
\end{align*}
where
\begin{align*}
&\theta_k(a,b,c,d)\\[1mm]
&\:=\frac{2k(1+2a-b-c-d+3k)}{a}+\frac{(a+k)(1+a-b-c+k)}{a(1+a-b+2k)}
\\[1mm]
&\quad\times\frac{(1+a-b-d+k)(1+a-c-d+k)(2+a-b-d+3k)}{(1+a-d+2k)(2+a-b-c-d+2k)}
\\[1mm]
&\:+\frac{(a+k)(c+k)(1-b+k)(1-d+k)}{a(1+2k)(1+a-b+2k)(1+a-c+2k)}
\\[1mm]
&\quad\times\frac{(1+a-b-c+k)(1+a-b-d+k)(1+a-c-d+k)}{(1+a-d+2k)(2+a-b-c-d+2k)}.
\end{align*}
The $(a,b,c)=(\frac{1}{2},\frac{1}{2},1-d)$ case of it reads
\begin{align}
\sum_{k=0}^{\infty}(-1)^k\frac{(\frac{1}{2})_k^4(d)_k^3(1-d)_k^3}{(1)_{2k}^3(\frac{1}{2}+d)_{2k}(\frac{3}{2}-d)_{2k}}\Omega_k(d)
=\frac{1-2d}{\pi}\tan(d\pi),\label{equation-bb}
\end{align}
where
\begin{align*}
\Omega_k(d) &=2k(1+6k)+\frac{(d+k)(1-d+k)(2-d+3k)}{3-2d+4k}
\\[1mm]
&\quad+\frac{(d+k)(1-d+k)^3}{(1+2d+4k)(3-2d+4k)}.
\end{align*}
Realize that the series in \eqref{equation-bb} is uniformly
convergent for $d\in \mathbb{C}$. Employ $\mathcal{D}_{d}$ on both
sides of \eqref{equation-bb} to gain
\begin{align*}
&\sum_{k=0}^{\infty}(-1)^k\frac{(\frac{1}{2})_k^4(d)_k^3(1-d)_k^3}{(1)_{2k}^3(\frac{1}{2}+d)_{2k}(\frac{3}{2}-d)_{2k}}\Omega_k(d)
\notag\\[1mm]
&\quad\times\bigg\{3H_{k}(d-1)-3H_{k}(-d)+H_{2k}\bigg(\frac{1}{2}-d\bigg)-H_{2k}\bigg(d-\frac{1}{2}\bigg)\bigg\}
\notag\\[1mm]
&\:+\sum_{k=0}^{\infty}(-1)^k\frac{(\frac{1}{2})_k^4(d)_k^3(1-d)_k^3}{(1)_{2k}^3(\frac{1}{2}+d)_{2k}(\frac{3}{2}-d)_{2k}}
\mathcal{D}_{d}\,\Omega_k(d)
\notag\\[1mm]
&\:\:=(1-2d)\sec^2(d\pi)-\frac{2}{\pi}\tan(d\pi).
\end{align*}
Dividing both sides of the last equation by $(1-2d)$, we have
\begin{align}
&\sum_{k=0}^{\infty}(-1)^k\frac{(\frac{1}{2})_k^4(d)_k^3(1-d)_k^3}{(1)_{2k}^3(\frac{1}{2}+d)_{2k}(\frac{3}{2}-d)_{2k}}\Omega_k(d)
\notag\\[1mm]
&\quad\times\bigg\{\sum_{i=1}^k\frac{3}{(d-1+i)(-d+i)}-\sum_{i=1}^{2k}\frac{1}{(d-\frac{1}{2}+i)(\frac{1}{2}-d+i)}\bigg\}
\notag\\[1mm]
&\:+\sum_{k=0}^{\infty}(-1)^k\frac{(\frac{1}{2})_k^4(d)_k^3(1-d)_k^3}{(1)_{2k}^3(\frac{1}{2}+d)_{2k}(\frac{3}{2}-d)_{2k}}
\frac{\mathcal{D}_{d}\,\Omega_k(d)}{1-2d}
\notag\\[1mm]
&\:\:=\sec^2(d\pi)-\frac{2}{\pi}\frac{\tan(d\pi)}{1-2d}.
\label{equation-cc}
\end{align}
Letting $a\to\frac{1}{2}$ in \eqref{equation-cc} and utilizing
\eqref{rule}, we catch hold of \eqref{equation-wei-h}.
\end{proof}



\begin{thebibliography}{99}
\small \setlength{\itemsep}{-.8mm}

\bibitem{Au} K.C. Au, Colored multiple zeta values, WZ-pairs and some infinite sums, preprint, arXiv: 2212. 02986v2.

\bibitem{Bailey} W.N. Bailey, Generalized Hypergeometric Series, Cambridge University Press, Cambridge,
1935.

\bibitem{Chu-b} W. Chu, W. Zhang, Accelerating Dougall's $_5F_4$-sum and infinite series involving $\pi$, Math
Comput. 285 (2014), 475--512.

\bibitem{Guo}V.J.W. Guo, X. Lian, Some $q$-congruences on double basic
hypergeometric sums, J. Difference Equ. Appl. 27 (2021), 453--461.

\bibitem{Guillera-a} J. Guillera, Some binomial series obtained by the WZ-method, Adv. Appl. Math. 29 (2002),
599--603.

\bibitem{Guillera-c}  J. Guillera, A new Ramanujan-like series for $1/\pi^2$, Ramanujan J. 26 (2011), 369--374.

\bibitem{Liu-a}Z.-G. Liu, Gauss summation and Ramanujan type series
for $1/\pi$, Int. J. Number Theory 8 (2012), 289--297.

\bibitem{Liu}
 H. Liu, W. Wang, Gauss's theorem and harmonic number summation formulae
with certain mathematical constants, J. Differ. Equ. Appl. 23
(2017), 1204--1218.

 \bibitem{Mao} G.-S. Mao, H. Pan, Congruences corresponding to hypergeometric identities I. $_2F_1$ transformations, J. Math. Anal. Appl. 505 (2022),
 Art. 125527.

\bibitem{Paule}
 P. Paule, C. Schneider, Computer proofs of a new family of harmonic number identities, Adv. Appl.
Math. 31 (2003), 359--378.

\bibitem{Sofo}
A. Sofo, Sums of derivatives of binomial coefficients, Adv. Appl.
Math. 42 (2009), 123--134.

\bibitem{Sun-11}Z.-W. Sun, Supercongruences and Euler numbers, Sci. China Math. 54
(2011), 2509--2535.

\bibitem{Sun-a}Z.-W. Sun, Supercongruences involving products of two binomial coefficients, Finite
Fields Appl. 22 (2013), 24--44.

\bibitem{Sun-b}Z.-W. Sun, Two $q$-analogues of Euler's formula $\zeta(2)=\pi^2/6$, Colloq. Math. 158 (2019),
313--320.

\bibitem{Sun}Z.-W. Sun, Series with summands involving harmonic numbers, preprint, arXiv: 2210. 07238v8.

\bibitem{Wang-Sun} C. Wang, Z.-W. Sun, Proof of some conjectural hypergeometric supercongruences
via curious identities,  J. Math. Anal. Appl. 505 (2022), Art.
125575.

\bibitem{Wang-Wei}
J. Wang, C. Wei, Derivative operator and summation formulae
involving generalized harmonic numbers, J. Math. Anal. Appl. 434
(2016), 315--341.

\bibitem{Wang} L. Wang, Y. Yang, Ramanujan-type $1/\pi$-series from bimodular
forms, Ramanujan J. 59 (2022), 831--882.









\end{thebibliography}
\end{document}